\documentclass[a4paper,12pt]{amsart}
\usepackage{amsmath,amsthm,amssymb,latexsym,epic}
\usepackage{graphicx,enumerate,stmaryrd}

\newtheorem{theorem}{Theorem}
\newtheorem{lemma}[theorem]{Lemma}

\newtheorem{corollary}[theorem]{Corollary}
\newtheorem{proposition}[theorem]{Proposition}
\newtheorem{example}[theorem]{Example}
\newtheorem{problem}[theorem]{Problem}
\usepackage[all]{xy}

\newcommand{\tto}{\twoheadrightarrow}

\begin{document}
\title[Selfadjoint functors]{On selfadjoint functors satisfying polynomial relations}
\author{Troels Agerholm and Volodymyr Mazorchuk}
\date{\today}

\maketitle

\begin{abstract}
We study selfadjoint functors acting on categories of finite 
dimensional modules over finite dimensional algebras with an
emphasis on functors satisfying some polynomial relations. 
Selfadjoint functors satisfying several easy relations, in 
particular, idempotents and square roots of a sum of
identity functors, are classified. We also describe various 
natural constructions for new actions using external direct 
sums, external tensor products, Serre subcategories,
quotients and centralizer subalgebras.
\end{abstract}

\section{Introduction}\label{s1}

The main motivation for the present paper stems from the recent
activities on categorification of representations of various algebras, 
see, in particular, \cite{CR,FKS,HS,HK,KMS1,MS1,MS2,MS3,R},
the review \cite{KMS2} and references therein. In these articles one could find
several results of the following kind: given a field $\Bbbk$, 
an associative $\Bbbk$-algebra $\Lambda$ with a fixed generating set 
$\{a_i\}$, and a $\Lambda$-module $M$, one constructs a 
{\em categorification} of $M$, that is an abelian category $\mathcal{C}$ 
and exact endofunctors $\{\mathrm{F}_i\}$ of $\mathcal{C}$ such that the
following holds: The Grothendieck group $[\mathcal{C}]$ of
$\mathcal{C}$ (with scalars extended to an appropriate field)  
is isomorphic to $M$ as a vector space and the functor $\mathrm{F}_i$ induces
on $[\mathcal{C}]$ the action of $a_i$ on $M$. Typical examples of 
algebras, for which categorifications of certain modules
are constructed, include group algebras of Weyl groups, Hecke algebras,
Schur algebras and enveloping algebras of some Lie algebras.
There are special reasons why such algebras and modules are of importance, 
for example, because of applications to link invariants (see \cite{St})
or Brou{\'e}'s abelian defect group conjecture (see \cite{CR}). 
Introducing some extra conditions one could even establish some 
uniqueness results, see \cite{CR,R}.

In this paper we would like to look at this problem from a different
perspective. The natural question, which motivates us, is the following: 
Given $\Lambda$ and $\{a_i\}$ can one classify {\em all} possible 
categorifications of {\em all} $\Lambda$-modules up to some 
natural equivalence? Of course in the full generality the problem is 
hopeless, as even the problem of classifying all $\Lambda$-modules 
seems hopeless for wild algebras. So, to start with, in this paper we 
make the main emphasis on the most basic example, that is the case
when the algebra $\Lambda$ is generated by one element, say $a$. 
If $\Lambda$ is finite-dimensional, then we necessarily have 
$f(a)=0$ for some nonzero $f(x)\in\Bbbk[x]$. To make the
classification problem more concrete, it is natural to look for 
a finite-dimensional $\Bbbk$-algebra $A$ and an exact endofunctor 
$\mathrm{F}$ of $A\text{-}\mathrm{mod}$, which should satisfy some 
sensible analogue  of the relation $f(\mathrm{F})=0$. 
Assume that all coefficients of $f(x)$ are
integral and rewrite $f(a)=0$ as $g(a)=h(a)$,
where both $g$ and $h$ have nonnegative coefficients. Setting
\begin{equation}\label{eq0}
\mathtt{k}\mathrm{F}:=
\begin{cases}
\underbrace{\mathrm{F}\oplus \mathrm{F}\oplus \cdots\oplus 
\mathrm{F}}_{\mathtt{k}
\text{ times}}, & \mathtt{k}\in \{1,2,\dots\};\\ 
\mathbf{0}, & \mathtt{k}=0;
\end{cases}
\end{equation}
and interpreting $+$ as $\oplus$, it makes sense to require 
$g(\mathrm{F})\cong h(\mathrm{F})$ for our functor $\mathrm{F}$.

To simplify our problem further we make another observation about the
examples of categorification available from the literature mentioned
above. All algebras appearing in this literature are equipped with
an involution, which in the categorification picture is interpreted
as ``taking the adjoint functor'' (both left- and right-adjoint). We again
take the simplest case of the trivial involution, and can now formulate
our main problem as follows:

\begin{problem}\label{prob1}
Given a finite dimensional $\Bbbk$-algebra $A$ and two polynomials 
$g(x)$ and $h(x)$ with nonnegative integral coefficients, classify, 
up to isomorphisms, all selfadjoint endofunctors $\mathrm{F}$ on
$A\text{-}\mathrm{mod}$ which satisfy $g(\mathrm{F})\cong h(\mathrm{F})$.
\end{problem}

Using Morita equivalence, in what follows we may assume that $A$ is basic
(i.e. has one dimensional simple modules). In this paper we obtain an answer 
to Problem~\ref{prob1} for relations $x^2=x$  (Section~\ref{s4}), 
$x^2=k$, $k\in\mathbb{Z}_+$, (Sections~\ref{s2}, \ref{s3} and
\ref{s9}) and $x^k=x^m$ (Section~\ref{s9}). For semisimple algebras 
Problem~\ref{prob1} reduces to solving certain matrix equations over 
matrices with nonnegative integer coefficients (Section~\ref{s5}). 

Another natural and important general question, which we address in this
paper, is how to produce new functorial actions by selfadjoint functors 
(e.g. new solutions to Problem~\ref{prob1}) from already known actions 
(known solutions to Problem~\ref{prob1}). In particular, in Section~\ref{s6} 
we describe the natural operations of external direct sums and external 
tensor products. In Section~\ref{s7} we study how functorial actions 
by selfadjoint functors  can be restricted to centralizer subalgebras.
In the special case of the algebra $A$ having the double
centralizer property for a projective-injective module $X$, we show
that there is a full and faithful functor from the category of 
selfadjoint functors on $A\text{-}\mathrm{mod}$ to the category of 
selfadjoint functors on $\mathrm{End}_A(X)^{\mathrm{op}}\text{-}\mathrm{mod}$.
We also present an example for which this functor is not dense
(essentially surjective).
Finally, in Section~\ref{s8} we study restriction of selfadjoint
functors to invariant Serre subcategories and induced actions on 
quotient categories, which we show are realized via induced actions 
on centralizer subalgebras.
\vspace{5mm}

\noindent
{\bf Acknowledgments.} For the second author the research was 
supported by the Royal Swedish Academy of Sciences and the 
Swedish Research Council. We thank Catharina Stroppel and Henning Haahr Andersen for some 
helpful discussions.

\section{Group actions on module categories}\label{s2}

Let $\Bbbk$ be an algebraically closed field,
$A$ a basic finite dimensional unital $\Bbbk$-algebra,
and $Z(A)$ the center of $A$.
All functors we consider are assumed to be additive and $\Bbbk$-linear.
We denote by $\mathbb{N}$ the set of positive integers and by 
$\mathbb{Z}_+$ the set of nonnegative integers.
Let $\{L_1,L_2,\dots,L_n\}$ be a complete list of pairwise nonisomorphic
simple $A$-modules. Let $P_i$, $i=1,\dots,n$,
denote the indecomposable projective cover of $L_i$.
We denote by $\mathrm{ID}$ the identity functor and by 
$\mathbf{0}$ the zero functor.

To start with we consider the easiest possible nontrivial equation 
\begin{equation}\label{eq1}
\mathrm{F}\circ\mathrm{F}\cong\mathrm{ID}, 
\end{equation}
which just means that $\mathrm{F}$ is a (covariant) involution on 
$A\text{-}\mathrm{mod}$. The answer to Problem~\ref{prob1} for relation 
\eqref{eq1} reduces to the following fairly well-known result
(for which we did not manage to find a reasonably explicit reference 
though):

\begin{proposition}\label{prop2}
\begin{enumerate}[(i)]
\item\label{prop2.1} For an algebra automorphism $\varphi:A\to A$ let 
${}_{\varphi}A$ denote the bimodule $A$ in which the 
left action is twisted by $\varphi$ (i.e. $a\cdot x\cdot b=\varphi(a)xb$). 
Then $\mathrm{F}_{\varphi}:={}_{\varphi}A\otimes_{A}{}_-$ 
is an autoequivalence of $A\text{-}\mathrm{mod}$.
\item\label{prop2.0} We have $\mathrm{F}_{\varphi}\circ 
\mathrm{F}_{\psi}\cong \mathrm{F}_{\varphi\circ \psi}$ for any
automorphisms $\varphi$ and $\psi$ of $A$.
\item\label{prop2.2} Every autoequivalence of $A\text{-}\mathrm{mod}$
is isomorphic to $\mathrm{F}_{\varphi}$  for some algebra 
automorphism $\varphi:A\to A$.
\item\label{prop2.3} $\mathrm{F}_{\varphi}\cong\mathrm{ID}$ if and only
if $\varphi$ is an inner automorphism.
\item\label{prop2.4} $\mathrm{F}_{\varphi}$ is selfadjoint if and only if
$\varphi^2$ is an inner automorphism. 
\end{enumerate}
\end{proposition}

\begin{proof}
The functor $\mathrm{F}_{\varphi}$ just twists the action of $A$ by
$\varphi$. This implies claim \eqref{prop2.0} and, in particular,
that $\mathrm{F}_{\varphi^{-1}}$ is an inverse to
$\mathrm{F}_{\varphi}$, which yields claim \eqref{prop2.1}.

Let $\mathrm{F}:A\text{-}\mathrm{mod}\to A\text{-}\mathrm{mod}$ be
an autoequivalence. Then $\mathrm{F}$ maps indecomposable
projectives to indecomposable projectives, in particular, 
$\mathrm{F}{}_AA\cong {}_AA$ and we can identify these modules fixing 
some isomorphism, say 
$\alpha:{}_AA\overset{\sim}{\longrightarrow} \mathrm{F}{}_AA$.
Let $\beta:A^{\mathrm{op}}\to \mathrm{End}_A({}_AA)$ be the natural
isomorphism sending $a$ to the right multiplication with $a$,
which we denote by $r_a$. Using the following sequence:
\begin{displaymath}
{}_AA\overset{\alpha}{\longrightarrow} 
\mathrm{F}{}_AA\overset{F(r_a)}{\longrightarrow}
\mathrm{F}{}_AA\overset{\alpha}{\longleftarrow}{}_AA
\end{displaymath}
we can define $\varphi(a):=\beta^{-1}(\alpha^{-1}F(r_a)\alpha)$.
Then $\varphi$ is an automorphism of $A$.
It is straightforward to verify that 
$\mathrm{F}\cong\mathrm{F}_{\varphi^{-1}}$. 
This proves claim \eqref{prop2.2}.
 
If $\varphi:A\to A$ is inner, say $\varphi(a)=sas^{-1}$ for some
invertible $s\in A$, it is straightforward to check that the
map $a\mapsto sa$ is a bimodule isomorphism from $A$ to 
${}_{\varphi}A$. This means that $\mathrm{F}_{\varphi}\cong\mathrm{ID}$
in this case. Conversely, if $\mathrm{F}_{\varphi}\cong\mathrm{ID}$,
then there is a bimodule isomorphism $f:A\to {}_{\varphi}A$. Let
$s=f(1)$. Then $s$ is invertible as $1\in f(A)=f(1)\cdot A=sA$ and 
$1\in f(A)=A\cdot f(1)=\varphi(A)s$. Also
$\varphi(a)s=a\cdot f(1)=f(a)=f(1)\cdot a=sa$, which yields 
$\varphi(a)=sas^{-1}$. This proves claim \eqref{prop2.3}.

As $\mathrm{F}_{\varphi}$ is an autoequivalence by claim \eqref{prop2.1}, 
the adjoint of $\mathrm{F}_{\varphi}$ is $\mathrm{F}_{\varphi^{-1}}$.
Therefore claim \eqref{prop2.4} follows from claim \eqref{prop2.3}
\end{proof}

We note that the bimodule ${}_{\varphi}A$ occurring in 
Proposition~\ref{prop2} is sometimes called a {\em twisted bimodule},
see for example \cite{EH}.

Let $G$ be a group. A {\em weak} (resp. {\em strong}) action of $G$ 
on $A\text{-}\mathrm{mod}$ is a collection $\{\mathrm{F}_g:g\in G\}$ of 
endofunctors of $A\text{-}\mathrm{mod}$ such that 
$\mathrm{F}_g\circ \mathrm{F}_h\cong \mathrm{F}_{gh}$ 
(resp. $\mathrm{F}_g\circ \mathrm{F}_h=\mathrm{F}_{gh}$) 
for all $g,h\in G$, and $\mathrm{F}_1\cong \mathrm{ID}$
(resp. $\mathrm{F}_1=\mathrm{ID}$). Two weak actions
$\{\mathrm{F}_g:g\in G\}$ and 
$\{\mathrm{F}'_g:g\in G\}$ are called {\em equivalent} provided that 
$\mathrm{F}_g\cong \mathrm{F}'_g$ for all $g\in G$. 
Let $\mathrm{Aut}(A)$ denote the
group of all automorphisms of $A$ and $\mathrm{Inn}(A)$ denote the normal
subgroup of $\mathrm{Aut}(A)$ consisting of all inner automorphisms.
Set $\mathrm{Out}(A):=\mathrm{Aut}(A)/\mathrm{Inn}(A)$.
From Proposition~\ref{prop2} we have:

\begin{corollary}\label{cor3}
Equivalence classes of weak actions of a group $G$ on 
$A\text{-}\mathrm{mod}$ are in one-to-one correspondence with group 
homomorphisms from $G$ to $\mathrm{Out}(A)$.
\end{corollary}

\begin{proof}
Let $\{\mathrm{F}_g:g\in G\}$  be a weak action of $G$ on 
$A\text{-}\mathrm{mod}$. Then for any $g\in G$ the functor $\mathrm{F}_g$ 
is an autoequivalence of $A\text{-}\mathrm{mod}$ and hence is isomorphic 
to the functor $\mathrm{F}_{\varphi_g}$ for some automorphism $\varphi_g$ 
of $A$ (Proposition~\ref{prop2}\eqref{prop2.2}). 
By Proposition~\ref{prop2}\eqref{prop2.3}, the 
automorphism $\varphi_g$ is defined up to a factor from 
$\mathrm{Inn}(A)$, hence, by Proposition~\ref{prop2}\eqref{prop2.0},
the map $g\mapsto \varphi_g\mathrm{Inn}(A)$ is a homomorphism from $G$ to 
$\mathrm{Out}(A)$. From the definitions it follows that equivalent actions 
produce the same homomorphism and nonequivalent actions produce different 
homomorphisms. The claim follows.
\end{proof}

\begin{corollary}\label{cor4}
If $\mathrm{Inn}(A)$ is trivial, then every weak action
of a group $G$ on $A\text{-}\mathrm{mod}$ is equivalent to a
strong action.
\end{corollary}

\begin{proof}
If  $\mathrm{Inn}(A)$ is trivial, the automorphism $\varphi_g$
from the proof of Corollary~\ref{cor3} is uniquely defined, so the
action $\{\mathrm{F}_g:g\in G\}$ is equivalent to the strong action 
$\{\mathrm{H}_{g}:g\in G\}$, where 
$\mathrm{H}_{g}$ denotes the functor of twisting the
$A$-action by $\varphi_g$. The claim follows.
\end{proof}

\begin{corollary}\label{cor5}
Isomorphism classes of selfadjoint functors $\mathrm{F}$ satisfying
\eqref{eq1} are in one-to-one correspondence with group homomorphisms
from $\mathbb{Z}_2$ to $\mathrm{Out}(A)$. The correspondence is given by:
\begin{equation}\label{eq3}
\begin{array}{c}
\mathrm{F}\mapsto f, \text{ where } f:\mathbb{Z}_2\to \mathrm{Out}(A) 
\text{ is such that }\\ 
\mathrm{F}\cong \mathrm{F}_{\varphi} \text{ for any } 
\varphi\in f(1).
\end{array}
\end{equation}
\end{corollary}

\begin{proof}
Note that any autoequivalence of $A\text{-}\mathrm{mod}$ satisfying
\eqref{eq1} is selfadjoint (as $\mathrm{F}\cong \mathrm{F}^{-1}$ by
\eqref{eq1}). Therefore the claim follows from Corollary~\ref{cor3} 
and its proof.
\end{proof}

\begin{corollary}\label{cor6}
\begin{enumerate}[(i)]
\item\label{cor6.1} Let $n\in\{2,3,4,\dots\}$. Then isomorphism classes 
of endofunctors $\mathrm{F}$ of $A\text{-}\mathrm{mod}$ satisfying
\begin{equation}\label{eq2}
\mathrm{F}^{n}:=\underbrace{\mathrm{F}\circ\mathrm{F}\circ \dots\circ
\mathrm{F}}_{\text{$n$ times}}\cong \mathrm{ID} 
\end{equation}
are in one-to-one correspondence with group homomorphisms
from $\mathbb{Z}_n$ to $\mathrm{Out}(A)$ (the correspondence is given 
by \eqref{eq3}, where $\mathbb{Z}_2$ is substituted by $\mathbb{Z}_n$).
\item\label{cor6.2} The endofunctor $\mathrm{F}$ from \eqref{cor6.1}
is selfadjoint if and only if $\mathrm{F}^2\cong \mathrm{ID}$.
\end{enumerate}
\end{corollary}

\begin{proof}
Claim \eqref{cor6.1} is proved similarly to Corollary~\ref{cor5}. 
Claim  \eqref{cor6.2} is obvious.
\end{proof}

\section{Selfadjoint functorial square roots}\label{s3}

In this section we consider the generalization 
\begin{equation}\label{eq4}
\mathrm{F}\circ\mathrm{F}\cong\mathtt{k}\mathrm{ID},\quad
\mathtt{k}\in\{2,3,4,\dots\}, 
\end{equation}
of the equation \eqref{eq1} (see \eqref{eq0} for notation). Our
main result here is the following:

\begin{theorem}\label{thm7}
\begin{enumerate}[(i)]
\item\label{thm7.1} A selfadjoint endofunctor $\mathrm{F}$
of $A\text{-}\mathrm{mod}$ satisfying \eqref{eq4} exists 
if and only if $\mathtt{k}=\mathtt{m}^2$ for some 
$\mathtt{m}\in\{2,3,4,\dots\}$.
\item\label{thm7.2} If $\mathtt{k}=\mathtt{m}^2$ for some 
$\mathtt{m}\in\{2,3,4,\dots\}$, then isomorphism classes of 
selfadjoint endofunctors  $\mathrm{F}$ on $A\text{-}\mathrm{mod}$ 
satisfying \eqref{eq4} are in one-to-one correspondence with 
isomorphism  classes of selfadjoint endofunctors  $\mathrm{F}'$ on $A\text{-}\mathrm{mod}$  satisfying \eqref{eq1}. The correspondence 
is  given by: $\mathtt{m}\mathrm{F}'\mapsto \mathrm{F}'$.
\end{enumerate}
\end{theorem}

Let $[A]$ denote the Grothendieck group 
of $A\text{-}\mathrm{mod}$. For $M\in A\text{-}\mathrm{mod}$ we denote by
$[M]$ the image of $M$ in $[A]$. The group $[A]$ is a free abelian group 
with basis $\mathbf{l}:=([L_1],[L_2],\dots,[L_n])$. Every exact endofunctor
$\mathrm{G}$ on $A\text{-}\mathrm{mod}$ defines a group 
endomorphism $[\mathrm{G}]$ of $[A]$. We denote by 
$M_{\mathrm{G}}$ the matrix of  $[\mathrm{G}]$
in the basis $\mathbf{l}$. Obviously, 
$M_{\mathrm{G}}\in \mathrm{Mat}_{n\times n}(\mathbb{Z}_+)$.

If $\mathrm{G}$ is selfadjoint, it is exact and maps projective modules
to projective modules (and injective modules to injective modules,
see for example \cite[Corollary~5.21]{Ma}).  Then 
$\mathrm{G}P_j=\oplus_{i=1}^n\mathtt{x}_{ij}P_i$. Define
$N_{\mathrm{G}}=(\mathtt{x}_{ij})_{i,j=1,\dots,n}$.

\begin{lemma}\label{lem9}
We have $N_{\mathrm{G}}=M_{\mathrm{G}}^t$, where ${\cdot}^t$ denotes
the {\em transposed} matrix.
\end{lemma}

\begin{proof}
Let $M_{\mathrm{G}}=(\mathtt{y}_{ij})_{i,j=1,\dots,n}$. The
claim follows from the selfadjointness of $\mathrm{G}$ as follows:
\begin{displaymath}
\mathtt{x}_{ij}=\dim \mathrm{Hom}_{A}(\mathrm{G}P_j,L_i)=
\dim \mathrm{Hom}_{A}(P_j,\mathrm{G}L_i)=\mathtt{y}_{ji}.\qedhere
\end{displaymath}
\end{proof}

To prove Theorem~\ref{thm7} we will need to understand 
$M_{\mathrm{F}}$ for selfadjoint functors
$\mathrm{F}$ satisfying \eqref{eq4}. Let $\mathbf{1}_n$ denote the 
identity matrix in $\mathrm{Mat}_{n\times n}(\mathbb{Z}_+)$.
Then from \eqref{eq4} we obtain
$M_{\mathrm{F}}^2=\mathtt{k}\cdot \mathbf{1}_n$. 
The canonical form for such $M_{\mathrm{F}}$ is given by the
following lemma:

\begin{lemma}\label{lem8}
Let $M\in\mathrm{Mat}_{n\times n}(\mathbb{Z}_+)$ be such that 
$M^2=\mathtt{k}\cdot \mathbf{1}_n$. 
Then there exists a permutation matrix $S$ such that
$SMS^{-1}$ is a direct sum of matrices of the form 
\begin{displaymath}
\left(\begin{array}{cc}0&a\\b&0\end{array}\right),\quad 
a,b\in \mathbb{Z}_+, ab=\mathtt{k};\quad\quad\text{ and }\quad\quad
\left(a\right),\quad a\in \mathbb{Z}_+, a^2=\mathtt{k}. 
\end{displaymath}
\end{lemma}

\begin{proof}
We proceed by induction on $n$. If $n=1$ the claim is obvious.
Let $M=(m_{ij})_{i,j=1,\dots,n}$
and $M^2=(k_{ij})_{i,j=1,\dots,n}$. If $m_{11}\neq 0$, then 
$m_{1j}=0$, $j=2,\dots,n$, for otherwise $k_{1j}\neq 0$
(as all our entries are in $\mathbb{Z}_+$). Similarly 
$m_{j1}=0$, $j=2,\dots,n$. This means that $M$ is a direct sum of
the block $(m_{11})$, where $m_{11}^2=\mathtt{k}$, 
and a matrix $\hat{M}$ of size $n-1\times n-1$ 
satisfying $\hat{M}^2=\mathtt{k}\cdot \mathbf{1}_{n-1}$.
The claim now follows from the inductive assumption.

If $m_{11}= 0$ then, since $k_{11}=\mathtt{k}$, there exists 
$j\in\{2,3,\dots,n\}$ such that $m_{1j}\neq 0$ and $m_{j1}\neq 0$. 
Substituting $M$ by $SMS^{-1}$, where $S$ is the transposition of
$j$ and $2$, we may assume $j=2$. Then from $k_{1j}=k_{j1}=0$
for all $j\neq 1$, and $k_{2j}=k_{j2}=0$ for all $j\neq 2$, it
follows that $m_{1j}=m_{j1}=0$ for all $j\neq 2$ and
$m_{2j}=m_{j2}=0$ for all $j\neq 1$. This means that $M$ is a direct 
sum of the block 
$\left(\begin{array}{cc}0&m_{1,2}\\m_{2,1}&0\end{array}\right)$, 
where $m_{12}m_{21}=\mathtt{k}$, and a matrix $\hat{M}$ of
size $n-2\times n-2$ satisfying $\hat{M}^2=\mathtt{k}\cdot \mathbf{1}_{n-2}$.
Again, the claim now follows from the inductive assumption. This completes
the proof.
\end{proof}

\begin{lemma}\label{lem12}
Assume that $\mathrm{F}$ is an
endofunctor on $A\text{-}\mathrm{mod}$ satisfying  \eqref{eq4}. Then  
$\mathrm{F}$ preserves the full subcategory $\mathcal{S}$ of 
$A\text{-}\mathrm{mod}$, which consists of semisimple $A$-modules.
\end{lemma}

\begin{proof}
As $\mathrm{F}$ is additive, to prove the claim we have to show
that $\mathrm{F}$ sends simple modules to semisimple modules. Since
$\mathrm{F}$ satisfies \eqref{eq4}, the matrix $M_{\mathrm{F}}$
satisfies $M_{\mathrm{F}}^2=\mathtt{k}\cdot \mathbf{1}_n$ and hence
is described by Lemma~\ref{lem8}. From the latter lemma 
it follows that for any $i\in\{1,2,\dots,n\}$ we 
have $[\mathrm{F}L_i]=a[L_j]$ for some $j\in\{1,2,\dots,n\}$ and
$a\in \mathbb{N}$, and, moreover, $[\mathrm{F}L_j]=b[L_i]$
for some $b\in\mathbb{N}$ such that $ab=\mathtt{k}$. Applying
$\mathrm{F}$ to any inclusion $L_i\hookrightarrow \mathrm{F}L_j$
we get $\mathrm{F}L_i \hookrightarrow \mathrm{F}\mathrm{F}L_j$.
However, $\mathrm{F}\mathrm{F}L_j\overset{\eqref{eq4}}{\cong} 
\mathtt{k}L_j$  is a semisimple module. Therefore $\mathrm{F}L_i$, 
being a submodule of a semisimple module, is semisimple itself.
\end{proof}

\begin{proof}[Proof of Theorem~\ref{thm7}\eqref{thm7.1}.]
If $\mathtt{k}=\mathtt{m}^2$ for some $\mathtt{m}\in\{2,3,4,\dots\}$,
then $\mathrm{F}=\mathtt{m}\mathrm{ID}$ is a selfadjoint functor
satisfying \eqref{eq4}. Hence to prove 
Theorem~\ref{thm7}\eqref{thm7.1} we have to show that  in the case
$\mathtt{k}\neq \mathtt{m}^2$ for any $\mathtt{m}\in\{2,3,4,\dots\}$
no selfadjoint $\mathrm{F}$ satisfies \eqref{eq4}.

In the latter case let us assume that $\mathrm{F}$ is a selfadjoint
endofunctor on $A\text{-}\mathrm{mod}$ satisfying \eqref{eq4}.
From Lemma~\ref{lem8} we have that, after a reordering
of simple modules, the matrix $M_{\mathrm{F}}$ becomes a direct sum of
matrices of the form $\left(\begin{array}{cc}0&a\\b&0\end{array}\right)$,
where $ab=\mathtt{k}$ and $a\neq b$. In particular, we have 
$[\mathrm{F}L_1]=\mathtt{b}[L_2]$ and $[\mathrm{F}L_2]=\mathtt{a}[L_1]$ 
for some $\mathtt{a},\mathtt{b}\in\mathbb{N}$ such that 
$\mathtt{a}\mathtt{b}=\mathtt{k}$ and $\mathtt{a}\neq \mathtt{b}$.
By Lemma~\ref{lem12}, we even have $\mathrm{F}L_1\cong \mathtt{b}L_2$ and 
$\mathrm{F}L_2\cong \mathtt{a}L_1$. Using this and the selfadjointness
of $\mathrm{F}$, we have:
\begin{displaymath}
\begin{array}{rcl} 
\mathtt{b}&=&\dim \mathrm{Hom}_A(\mathtt{b}L_2,L_2)\\
&=&\dim \mathrm{Hom}_A(\mathrm{F}L_1,L_2)\\
&=&\dim \mathrm{Hom}_A(L_1,\mathrm{F}L_2)\\
&=&\dim \mathrm{Hom}_A(L_1,\mathtt{a}L_1)\\
&=&\mathtt{a},
\end{array}
\end{displaymath}
a contradiction. The claim of Theorem~\ref{thm7}\eqref{thm7.1} follows.
\end{proof}

\begin{proof}[Proof of Theorem~\ref{thm7}\eqref{thm7.2}.]
This proof is inspired by \cite{LS}. Assume that 
$\mathtt{k}=\mathtt{m}^2$ for some $\mathtt{m}\in\{2,3,4,\dots\}$.
If $\mathrm{F}'$ is a selfadjoint endofunctor on 
$A\text{-}\mathrm{mod}$  satisfying \eqref{eq1}, then 
$\mathtt{m}\mathrm{F}'$ is a selfadjoint endofunctor on 
$A\text{-}\mathrm{mod}$  satisfying \eqref{eq4}. Hence to prove
Theorem~\ref{thm7}\eqref{thm7.2} we have to establish the
converse statement.

Let $\mathrm{F}$ be some selfadjoint endofunctor on
$A\text{-}\mathrm{mod}$ satisfying \eqref{eq4}. 
Our strategy of the proof is as follows: we would like to
show that the functor $\mathrm{F}$ decomposes into a direct sum
of $\mathtt{m}$ nontrivial functors and then use the results from
Section~\ref{s2} to get that these functors have the required form.
To prove decomposability of $\mathrm{F}$ we produce $\mathtt{m}$
orthogonal idempotents in the endomorphism ring of $\mathrm{F}$.
For this we first show that the necessary idempotents exist in the
case of a semisimple algebra, and then use lifting of idempotents
modulo the radical. All the above requires some preparation and
technical work. 

From Lemma~\ref{lem8} and the above proof of 
Theorem~\ref{thm7}\eqref{thm7.1} it follows that, re-indexing,
if necessary, simple $A$-modules, the matrix $M_{\mathrm{F}}$
reduces to a direct sum of the blocks  
\begin{equation}\label{eq7}
\left(\begin{array}{cc}0&\mathtt{m}\\\mathtt{m}&0\end{array}\right)
\quad\quad \text{ and/or }\quad\quad \left(\mathtt{m}\right). 
\end{equation}

\begin{lemma}\label{lem11}
The claim of Theorem~\ref{thm7}\eqref{thm7.2} is true in the
case of a semisimple algebra $A$.
\end{lemma}

\begin{proof}
Assume first that $A$ is semisimple (and basic). 
Then $A\cong \oplus_{i=1}^n\Bbbk$ 
and $A\text{-}\mathrm{mod}\cong \oplus_{i=1}^n\Bbbk\text{-}\mathrm{mod}$.
The only (up to isomorphism) indecomposable nonzero functor from
$\Bbbk\text{-}\mathrm{mod}$ to $\Bbbk\text{-}\mathrm{mod}$ is the
identity functor (as $\Bbbk\otimes \Bbbk\cong\Bbbk$). Therefore from
\eqref{eq7} we get that $\mathrm{F}$ is isomorphic to a direct sum of
functors of the form
\begin{displaymath}
\xymatrix{
\Bbbk\text{-}\mathrm{mod}\ar@/^/[rr]^{\mathtt{m}\mathrm{ID}}&&
\Bbbk\text{-}\mathrm{mod}\ar@/^/[ll]^{\mathtt{m}\mathrm{ID}}
} 
\quad\quad\text{ and/or }\quad\quad
\xymatrix{
 \ar@(ul,ur)[]^{\mathtt{m}\mathrm{ID}} 
\Bbbk\text{-}\mathrm{mod}
}
\end{displaymath}
(corresponding to the blocks from \eqref{eq7}). Define $\mathrm{F}'$ 
as the corresponding direct sum of functors of the form
\begin{displaymath}
\xymatrix{
\Bbbk\text{-}\mathrm{mod}\ar@/^/[rr]^{\mathrm{ID}}&&
\Bbbk\text{-}\mathrm{mod}\ar@/^/[ll]^{\mathrm{ID}}
} 
\quad\quad\text{ and/or }\quad\quad
\xymatrix{
 \ar@(ul,ur)[]^{\mathrm{ID}} 
\Bbbk\text{-}\mathrm{mod}.
}
\end{displaymath}
Then $\mathrm{F}'$ is selfadjoint and obviously satisfies \eqref{eq1},
moreover,  $\mathrm{F}\cong \mathtt{m}\mathrm{F}'$. This proves 
Theorem~\ref{thm7}\eqref{thm7.2} in the case of a semisimple
algebra $A$.
\end{proof}

Let $V$ be an $A\text{-}A$-bimodule such that $\mathrm{F}\cong
V\otimes_{A}{}_-$.  The right adjoint of $\mathrm{F}$ is $\mathrm{F}$ 
itself, in particular, this right adjoint is an exact functor
and hence is given by tensoring with 
the bimodule $\mathrm{Hom}_{A-}(V,A)$ (that is 
$V\cong \mathrm{Hom}_{A-}(V,A)$). The bimodule $V$ is projective both 
as a right $A$-module (as $\mathrm{F}$ is exact) and as a left $A$-module
(as $\mathrm{F}$ sends projective modules to projective modules). Hence
we have an isomorphism of $A\text{-}A$-bimodules as follows:
\begin{displaymath}
\begin{array}{ccc}
\mathrm{Hom}_{A-}(V,A)\otimes_A V & \cong & \mathrm{Hom}_{A-}(V,V)\\
f\otimes v & \mapsto & (w\mapsto f(w)v).
\end{array}
\end{displaymath}
This gives us the following isomorphism of $A\text{-}A$-bimodules:
\begin{displaymath}
V\otimes_A V \cong \mathrm{Hom}_{A-}(V,A)\otimes_A V
\cong  \mathrm{Hom}_{A-}(V,V).
\end{displaymath}
Note that the functor $\mathrm{F}\circ\mathrm{F}$ is given by
tensoring with the bimodule $V\otimes_A V\cong \mathrm{Hom}_{A-}(V,V)$. 
From \eqref{eq4} we thus get an isomorphism 
\begin{equation}\label{eq8}
\mathrm{Hom}_{A-}(V,V)\cong \mathtt{k}A
\end{equation}
of $A$-bimodules. Taking on both sides of the latter isomorphism elements 
on which the left and the right actions of $A$ coincide, we get an 
isomorphism
\begin{equation}\label{eq9}
\mathrm{Hom}_{A-A}(V,V)\cong \mathtt{k}Z(A)
\end{equation}
of $Z(A)$-bimodules.

Let $R_l$ and $R_r$ denote the radical of $V$, considered as a left and
as a right $A$-module, respectively. As $V\otimes_{A}{}_-$ sends simple
modules to semisimple modules (Lemma~\ref{lem12}), it follows that 
$R_l=R_r=:R$.

From \eqref{eq7} we have that the matrix $M_{\mathrm{F}}$ is symmetric.
Hence $N_{\mathrm{F}}=M_{\mathrm{F}}$ (Lemma~\ref{lem9}). Therefore from
\eqref{eq7} it follows that each indecomposable projective module
occurs in $\mathrm{F}{}_AA$ with multiplicity $\mathtt{m}$, that is
$\mathrm{F}{}_AA\cong \mathtt{m}\, {}_AA$. Using decomposition
\eqref{eq8} we thus can choose a basis $\{b_i:i=1,\dots,\mathtt{k}\}$ 
of $\mathrm{Hom}_{A-}(V,V)$ as a free left $A$-module such that the
left and the right actions of $A$ on the elements of this basis
coincide. Then all $b_i$'s belong to $\mathrm{Hom}_{A-A}(V,V)$ and form
there a basis as a free $Z(A)$-module (both left and right).

\begin{lemma}\label{lem14}
Let $L=\oplus_{i=1}^nL_i$ and $b$ be a nontrivial $\Bbbk$-linear 
combination of $b_i$'s. Then $b$ is a natural transformation from 
$\mathrm{F}$ to $\mathrm{F}$ and the induced endomorphism $b_{L}$ 
of the $A$-module $\mathrm{F}L$ is nonzero.
\end{lemma}

\begin{proof}
Applying the exact functor $V\otimes_A{}_-$ to the short exact sequence
\begin{displaymath}
0\to\mathrm{Rad}(A)\to  {}_AA\to L\to 0,
\end{displaymath}
we obtain the short exact sequence
\begin{displaymath}
0\to V\otimes_A\mathrm{Rad}(A)\to  V\to V\otimes_A L\to 0.
\end{displaymath}
Note that $V\otimes_A\mathrm{Rad}(A)=R$.
Applying to the latter sequence
the exact functor $\mathrm{Hom}_{A-}(V,{}_-)$
we obtain the short exact sequence
\begin{displaymath}
0\to \mathrm{Hom}_{A-}(V,R)\to  
\mathrm{Hom}_{A-}(V,V)\to \mathrm{Hom}_{A-}(V,V\otimes_A L)\to 0.
\end{displaymath}
By the definition of $b$, the image of $b\in \mathrm{Hom}_{A-}(V,V)$ 
does not belong to $R$ and hence $b$ induces a 
nonzero element $\overline{b}\in\mathrm{Hom}_{A-}(V,V\otimes_A L)$. By 
adjunction we have the following isomorphism:
\begin{displaymath}
\mathrm{Hom}_{A}(L,\mathrm{Hom}_{A-}(V,V\otimes_A L))
\cong
\mathrm{Hom}_{A}(V\otimes_A L,V\otimes_A L),
\end{displaymath}
which produces the nonzero endomorphism $b_L$ of $V\otimes_A L$ from 
our nonzero element $\overline{b}$. The claim follows.
\end{proof}

Set $\overline{A}=A/\mathrm{Rad}(A)$, then $\overline{A}$ is a semisimple
algebra. The bimodule $\overline{V}=V/R$ is an $\overline{A}$-bimodule
satisfying \eqref{eq4} and we have  $\mathrm{Rad}(A)\overline{V}=\overline{V}\mathrm{Rad}(A)=0$. Hence
we have the quotient homomorphism of algebras as follows:
\begin{displaymath}
\Phi:\mathrm{Hom}_{A-A}(V,V)\to 
\mathrm{Hom}_{\overline{A}-\overline{A}}(\overline{V},\overline{V}). 
\end{displaymath}
Note that the algebra $\overline{A}=Z(\overline{A})\cong n\Bbbk$ 
contains as a subalgebra the algebra $\overline{Z(A)}:= 
Z(A)/\mathrm{Rad}(Z(A))$. The algebra $\overline{Z(A)}$ is isomorphic
to $k\Bbbk$, where $k$ is the number of connected components of
the algebra $A$.  In particular, $\overline{Z(A)}\cong \overline{A}$
if $A$ is a direct sum of local algebras.

\begin{lemma}\label{lem15}
The kernel of $\Phi$ is the radical of $\mathrm{Hom}_{A-A}(V,V)$
and the image of $\Phi$ is isomorphic to the algebra 
$\mathrm{Mat}_{\mathtt{m}\times \mathtt{m}}(\overline{Z(A)})$.
\end{lemma}

\begin{proof}
The space $\mathrm{Hom}_{A-A}(V,V)$ is a free left $Z(A)$-module
of rank $\mathtt{k}$ with the basis $\{b_i\}$ from the above.
Since $\overline{A}$ is semisimple, so is $\overline{V}$ and thus
both $\mathrm{Rad}(Z(A))$ and the radical of
$\mathrm{Hom}_{A-A}(V,V)$ annihilate $\overline{V}$ both from the
left and from the right. Therefore the first claim of the lemma 
follows from the second claim just by counting dimensions.

Similarly to the proof of Lemma~\ref{lem14} one shows that the image $X$ 
of  $\Phi$ has dimension $\mathtt{k}\cdot \dim\overline{Z(A)}$
and is a subalgebra of 
$\mathrm{End}_{\overline{A}}(\overline{V}\otimes_{\overline{A}}
\overline{A})$, which corresponds to the embedding
$\overline{Z(A)}\subset \overline{A}$ (the algebra
$\mathrm{End}_{\overline{A}}(\overline{V}\otimes_{\overline{A}}
\overline{A})$ is free both as a left and as a right
$\overline{A}$-module).
By Lemma~\ref{lem11} we have an isomorphism 
$\overline{V}\otimes_{\overline{A}}\overline{A}\cong
\overline{V}\cong \mathtt{m}\overline{A}$ of $\overline{A}$-modules.
Hence the algebra $X$ is a subalgebra of the algebra  
\begin{displaymath}
\mathrm{End}_{\overline{A}}(\mathtt{m}\overline{A}) \cong
\mathrm{Mat}_{\mathtt{m}\times \mathtt{m}}(\overline{A}),
\end{displaymath}
which corresponds to the embedding $\overline{Z(A)}\subset \overline{A}$.
This means that 
$X\cong \mathrm{Mat}_{\mathtt{m}\times \mathtt{m}}(\overline{Z(A)})$.
\end{proof}

We have 
$\mathrm{Mat}_{\mathtt{m}\times \mathtt{m}}(\overline{Z(A)})
\cong  \mathrm{Mat}_{\mathtt{m}\times \mathtt{m}}(\Bbbk)\otimes
\overline{Z(A)}$.
Let $e_i$, $i=1,\dots,n$, denote the usual primitive diagonal idempotents 
of $\mathrm{Mat}_{\mathtt{m}\times \mathtt{m}}(\Bbbk)$ such that 
$\sum_i e_i$ is the identity matrix. By Lemma~\ref{lem15} we can lift 
the idempotents $e_i\otimes 1$ from 
$\mathrm{Mat}_{\mathtt{m}\times \mathtt{m}}(\Bbbk)
\otimes \overline{Z(A)}$ to $\mathrm{Hom}_{A-A}(V,V)$ modulo 
the radical (see e.g. \cite[3.6]{La}). Thus we obtain 
$\mathtt{m}$ orthogonal idempotents in $\mathrm{Hom}_{A-A}(V,V)$,
which implies the existence of a decomposition
\begin{displaymath}
\mathrm{F}=\mathrm{F}_1\oplus \mathrm{F}_2\oplus \dots
\oplus \mathrm{F}_{\mathtt{m}}
\end{displaymath}
for the functor $\mathrm{F}$. As $\overline{Z(A)}$ is a unital subalgebra 
of $\overline{A}$, we have an isomorphism 
$(e_i\otimes 1)\overline{A}\cong \overline{A}$ of left
$\overline{A}$-modules. Hence, it follows that
\begin{equation}\label{eq15}
\mathrm{F}_i L\cong L\quad \text{ for all }\quad i\in\{1,2,\dots,n\}.
\end{equation}

From \eqref{eq4} we have $\sum_{i,j}\mathrm{F}_i\circ\mathrm{F}_j\cong
\mathtt{k}\mathrm{ID}$. From \eqref{eq15} and the Krull-Schmidt theorem
it follows that $\mathrm{F}_i\circ\mathrm{F}_j\cong\mathrm{ID}$
for every $i$ and $j$. In particular, 
$\mathrm{F}_i\circ\mathrm{F}_i\cong\mathrm{ID}$, which yields that
every $\mathrm{F}_i$ is selfadjoint by 
Proposition~\ref{prop2}\eqref{prop2.4}.
Now the claim of Theorem~\ref{thm7}\eqref{thm7.2}
follows from Proposition~\ref{prop2}. This completes the proof.
\end{proof}

\section{External direct sums and tensor products}\label{s6}

To construct new solutions to functorial equations one may use
the classical constructions of external direct sums and tensor products.

We start with the construction of an external direct sum.
Let $g(x),h(x)\in\mathbb{Z}_+(x)$. Assume that for $i=1,2$ we have
a finite dimensional associative $\Bbbk$-algebra $A_i$ and 
a (selfadjoint) exact functor $\mathrm{F}_i$ on $A_i\text{-}\mathrm{mod}$
such that  $g(\mathrm{F}_i)\cong h(\mathrm{F}_i)$.
Set $A=A_1\oplus A_2$ and let $\mathrm{F}:=\mathrm{F}_1\boxplus
\mathrm{F}_2$ denote the external direct sum of $\mathrm{F}_1$
and $\mathrm{F}_2$ (it acts on $A$-modules componentwise).

\begin{proposition}\label{prop65}
The functor $\mathrm{F}$ is a (selfadjoint) exact endofunctor 
on $A\text{-}\mathrm{mod}$ satisfying $g(\mathrm{F})\cong h(\mathrm{F})$.
\end{proposition}

\begin{proof}
The action of $\mathrm{F}$ is computed componentwise and hence
properties of $\mathrm{F}$ follow from the corresponding properties of
the $\mathrm{F}_i$'s.
\end{proof}

The external tensor product works as follows:
Let $g(x),h(x)\in\mathbb{Z}_+(x)$. Assume that we have
a finite dimensional associative $\Bbbk$-algebra $A$ and 
a (selfadjoint) exact functor $\mathrm{F}$ on $A\text{-}\mathrm{mod}$
such that  $g(\mathrm{F})\cong h(\mathrm{F})$. Let $B$ be a
finite dimensional associative $\Bbbk$-algebra and $\mathrm{ID}_B$ 
denote the identity functor on $B\text{-}\mathrm{mod}$.
Consider the algebra $C=A\otimes B$. Then the external tensor product
$\mathrm{G}:=\mathrm{F}\boxtimes\mathrm{ID}_B$ is an 
exact endofunctor on $C\text{-}\mathrm{mod}$ defined as follows:
Any $X\in C\text{-}\mathrm{mod}$ can be considered as an $A$-module
with a fixed action of $B$ by endomorphisms $\psi_b$, $b\in B$. Then 
the $C$-module $\mathrm{G}X$ is defined as the $A$-module $\mathrm{F}X$ 
with the action of $B$ given by $\mathrm{F}\psi_b$. The action of
$\mathrm{G}$ on morphisms is defined in the natural way.

\begin{proposition}\label{prop66}
\begin{enumerate}[(i)]
\item\label{prop66.1} The functor $\mathrm{G}$ is selfadjoint if and only
if $\mathrm{F}$ is selfadjoint.
\item\label{prop66.2} There is an isomorphism of functors as follows:
$g(\mathrm{G})\cong h(\mathrm{G})$.
\end{enumerate}
\end{proposition}

\begin{proof}
From the definition of $\mathrm{G}$ it follows that the
adjunction morphisms $\mathrm{adj}:\mathrm{ID}_A\to \mathrm{F}\mathrm{F}$ and
$\mathrm{adj}':\mathrm{F}\mathrm{F}\to \mathrm{ID}_A$ induce in the natural
way adjunction morphisms $\overline{\mathrm{adj}}:\mathrm{ID}_C\to \mathrm{G}\mathrm{G}$ and $\overline{\mathrm{adj}}':\mathrm{G}\mathrm{G}
\to \mathrm{ID}_C$, and vice versa. This proves claim \eqref{prop66.1}.

Any isomorphism $g(\mathrm{F})\cong h(\mathrm{F})$ of functors induces,
by the functoriality of $\mathrm{F}$ and the definition of $\mathrm{G}$,
an isomorphism $g(\mathrm{G})\cong h(\mathrm{G})$. Claim
\eqref{prop66.2} follows and the proof is complete.
\end{proof}

\section{Selfadjoint idempotents (orthogonal projections)}\label{s4}

In this section we consider the equation
\begin{equation}\label{eq5}
\mathrm{F}\circ\mathrm{F}\cong\mathrm{F}, 
\end{equation}
which simply means that $\mathrm{F}$ is a selfadjoint idempotent 
(an orthogonal projection). For the theory of $*$-representations
of algebras, generated by orthogonal projections, we refer the reader
to \cite{Co,KS,KRS} and references therein. 

Every decomposition $A\cong B\oplus C$ into a direct sum of algebras 
(unital or zero) yields a decomposition
$A\text{-}\mathrm{mod}=B\text{-}\mathrm{mod}\oplus 
C\text{-}\mathrm{mod}$. Denote by $\mathfrak{p}_B:A\text{-}\mathrm{mod}\to 
B\text{-}\mathrm{mod}$ the natural projection with respect to this 
decomposition, that is the functor $\mathrm{ID}\boxplus\mathbf{0}$. 
Our main result in this section is the following:

\begin{theorem}\label{thm51}
Assume that $\mathrm{F}$ is a selfadjoint endofunctor on 
$A\text{-}\mathrm{mod}$ satisfying \eqref{eq5}. Then there exists 
a decomposition $A\cong B\oplus C$ into a direct sum of algebras 
(unital or zero) such that $\mathrm{F}\cong \mathfrak{p}_B$.
\end{theorem}

\begin{proof}
For a simple $A$-module $L$ set $\mathrm{F}L=X_L$.

\begin{lemma}\label{lem53}
We have $X_L=0$ or $X_L\cong L\oplus Y_L$ such that $\mathrm{F}Y_L=0$.
\end{lemma}

\begin{proof}
Assume $X_L\neq 0$. Then we have 
\begin{displaymath}
\begin{array}{rcl} 
0&\neq &\mathrm{Hom}_{A}(X_L,X_L)\\
&= &\mathrm{Hom}_{A}(\mathrm{F}L,\mathrm{F}L)\\
(\text{by adjunction})&= &\mathrm{Hom}_{A}(L,\mathrm{F}\mathrm{F}L)\\
(\text{by \eqref{eq5}})&= &\mathrm{Hom}_{A}(L,\mathrm{F}L)\\
&= &\mathrm{Hom}_{A}(L,X_L).
\end{array}
\end{displaymath}
Similarly, $\mathrm{Hom}_{A}(X_L,L)\neq 0$, which means that $L$ is both,
a submodule and a quotient of $X_L$. In particular, we have 
$[X_L]=[L]+z$ for some $z\in[A]$.

Further, we have
\begin{multline*}
[X_L]=[\mathrm{F}L]\overset{\eqref{eq5}}{=}[\mathrm{F}^2L]=
[\mathrm{F}X_L]=[\mathrm{F}][X_L]=\\=[\mathrm{F}]([L]+z)=
[\mathrm{F}][L]+[\mathrm{F}]z=
[\mathrm{F}L]+[\mathrm{F}]z=[X_L]+[\mathrm{F}]z.
\end{multline*}
This yields $[\mathrm{F}]z=0$. In particular, $L$ occurs with
multiplicity one in $X_L$ and hence, by the previous paragraph, 
$X_L\cong L\oplus Y_L$ for some $Y_L$. We further have $[Y_l]=z$ and
thus $\mathrm{F}Y_L=0$ follows from $[\mathrm{F}]z=0$. This completes
the proof.
\end{proof}

\begin{lemma}\label{lem54}
For every $L$ such that $\mathrm{F}L\neq 0$ we have $Y_L=0$.
\end{lemma}

\begin{proof}
Assume that $Y_L\neq 0$ and let $L'$ be a simple submodule of $Y_L$. 
Then $\mathrm{F}L'=0$ by Lemma~\ref{lem53}, in particular, $L'\neq L$. 
Hence we have
\begin{displaymath}
\begin{array}{rcl} 
0&\neq&\mathrm{Hom}_{A}(L',Y_L)\\
&=&\mathrm{Hom}_{A}(L',Y_L\oplus L)\\
&=&\mathrm{Hom}_{A}(L',\mathrm{F}L)\\
(\text{by adjunction})&=&\mathrm{Hom}_{A}(\mathrm{F}L',L)\\
(\text{by Lemma~\ref{lem53}})&=&0.
\end{array}
\end{displaymath}
The obtained contradiction completes the proof.
\end{proof}

From Lemmata~\ref{lem53} and \ref{lem54} it follows that the
matrix $M_{\mathrm{F}}$ is diagonal with zeros and ones on the diagonal,
and the ones correspond to exactly those simple $A$-modules, which are not
annihilated by $\mathrm{F}$. Without loss of generality we may assume
that the simple modules not annihilated by $\mathrm{F}$ are $L_1,L_2,\dots,L_k$
for some $k\in\{0,1,\dots,n\}$. From Lemma~\ref{lem9} we also obtain
$M_{\mathrm{F}}=N_{\mathrm{F}}$, that is
\begin{displaymath}
\mathrm{F}P_i=
\begin{cases}
P_i,& i\leq k;\\
0,& i>k. 
\end{cases}
\end{displaymath}
As any simple module is sent to a simple module or zero, it follows that for
$i\leq k$ all simple subquotients of $P_i$ have the form
$L_j$, $j\leq k$; and for $i>k$ all simple subquotients of $P_i$ have 
the form $L_j$, $j>k$. Therefore there is a decomposition
$A\cong B\oplus C$, where 
\begin{displaymath}
B\cong\mathrm{End}_{A}(P_1\oplus P_2\oplus\dots\oplus P_k)^{\mathrm{op}},
\quad C\cong\mathrm{End}_{A}(P_{k+1}\oplus P_{k+2}\oplus\dots\oplus P_n)^{\mathrm{op}}.
\end{displaymath}
The adjunction morphism $\mathrm{ID}\to \mathrm{F}^2\cong \mathrm{F}$
is nonzero on all $L_i$, $i\leq k$, and hence is an isomorphism
as $\mathrm{F}L_i\cong L_i$ by Lemmata~\ref{lem53} and \ref{lem54}.
By induction on the length of a module and the Five Lemma it
follows that the adjunction morphism is an isomorphism on all
$B$-modules (see e.g. \cite[3.7]{Ma} for details). Therefore $\mathrm{F}$
is isomorphic to the identity functor, when restricted to 
$B\text{-}\mathrm{mod}$. By the definition of $C$, the functor
$\mathrm{F}$ is the zero functor on $C\text{-}\mathrm{mod}$.
The claim of the theorem follows.
\end{proof}

\begin{corollary}\label{cor52}
If $A$ is connected then the only selfadjoint solutions to \eqref{eq5} 
are the identity and the zero functors.
\end{corollary}

\begin{proof}
If $A$ is connected and $A\cong B\oplus C$, then either $B$ or
$C$ is zero. Thus the statement follows  directly from 
Theorem~\ref{thm51}.
\end{proof}

\begin{corollary}\label{cor53}
If $\mathrm{F}$ and $\mathrm{G}$ are two selfadjoint solutions to 
\eqref{eq5}, then $\mathrm{F}\circ\mathrm{G}\cong \mathrm{G}\circ\mathrm{F}$.
\end{corollary}

\begin{proof}
Define:
\begin{displaymath}
\begin{array}{rcl} 
X_{00}&=&\{i\in\{1,2,\dots,n\}:\mathrm{F}L_i\neq 0,\mathrm{G}L_i\neq 0\},\\
X_{10}&=&\{i\in\{1,2,\dots,n\}:\mathrm{F}L_i= 0,\mathrm{G}L_i\neq 0\},\\
X_{01}&=&\{i\in\{1,2,\dots,n\}:\mathrm{F}L_i\neq 0,\mathrm{G}L_i= 0\},\\
X_{11}&=&\{i\in\{1,2,\dots,n\}:\mathrm{F}L_i=0,\mathrm{G}L_i=0\}.
\end{array}
\end{displaymath}
Then $\{1,2,\dots,n\}$ is a disjoint union of $X_{ij}$, $i,j\in\{0,1\}$.
For $i,j\in\{0,1\}$ set
\begin{displaymath}
A_{ij}:=\mathrm{End}_{A}(\oplus_{s\in X_{ij}}P_s)^{\mathrm{op}}.
\end{displaymath}
Similarly to the proof of Theorem~\ref{thm51} one obtains that
$A\cong \oplus_{i,j=0}^1A_{i,j}$ and, moreover, that both
$\mathrm{F}\circ\mathrm{G}$ and $\mathrm{G}\circ\mathrm{F}$
are isomorphic to $\mathfrak{p}_{A_{00}}$ with respect to this
decomposition. The claim follows.
\end{proof}

Selfadjointness of $\mathrm{F}$ is important for the claim of
Theorem~\ref{thm51}. Here is an example of an exact, but not
selfadjoint, functor satisfying \eqref{eq5}, which is not of the
type $\mathfrak{p}_{B}$: Let $A=\Bbbk\oplus \Bbbk$, then
an $A$-module is just a collection $(X,Y)$ of two vector spaces.
Define the functor $\mathrm{F}$ as follows:
$\mathrm{F}(X,Y):=(X\oplus Y,0)$ with the obvious action on morphisms.
Then $\mathrm{F}$ satisfies \eqref{eq5} but is not selfadjoint.
In fact, for any idempotent matrix 
$M\in\mathrm{Mat}_{n\times n}(\mathbb{Z}_+)$ one can similarly define 
an exact endofunctor $\mathrm{F}$ on $A\text{-}\mathrm{mod}$, where
\begin{displaymath}
A=\underbrace{\Bbbk\oplus \Bbbk\oplus\dots
\oplus \Bbbk}_{\text{$n$ summands}},
\end{displaymath}
such that  $M_{\mathrm{F}}=M$ (see Section~\ref{s5} for more details).

There are also many natural idempotent functors, which are exact on
only one side. For example, for any $X\subset \{1,2,\dots,n\}$ one
could define an idempotent right exact (but, in general, not left 
exact) endofunctor $\mathrm{Z}_X$ on $A\text{-}\mathrm{mod}$ as follows:
$\mathrm{Z}_XN$ is the maximal quotient of $N$, whose simple
subquotients are all isomorphic to $L_i$, $i\in X$. The latter 
functors appear in Lie Theory, see e.g. \cite{MS9}.
 
\section{Functors generating a cyclic semigroup}\label{s9}

\begin{proposition}\label{prop92}
Let $\mathrm{F}$ be a selfadjoint endofunctor on 
$A\text{-}\mathrm{mod}$. If $\mathrm{F}^k=0$ for some
$k\in\mathbb{N}$, then $\mathrm{F}=0$.
\end{proposition}

\begin{proof}
The claim is obvious for $k=1$. Assume that $k=2$. Then
$\mathrm{F}^2=0$. The condition $\mathrm{F}\neq 0$ is equivalent to
the condition $\mathrm{F}L_i\neq 0$ for some $i\in\{1,2,\dots,n\}$.
If $\mathrm{F}L_i\neq 0$, then, using the adjunction, we get
\begin{displaymath}
0\neq \mathrm{Hom}_A(\mathrm{F}L_i,\mathrm{F}L_i) \cong
\mathrm{Hom}_A(L_i,\mathrm{F}\mathrm{F}L_i).
\end{displaymath}
However, $\mathrm{F}\mathrm{F}L_i=0$ as $\mathrm{F}^2=0$, a contradiction.
Therefore $\mathrm{F}=0$.

Now we proceed by induction on $k$. Assume $k>2$. Then 
$\mathrm{F}^k=0$ implies $\mathrm{F}^{2(k-1)}=(\mathrm{F}^{k-1})^2=0$.
As $\mathrm{F}^{k-1}$ is selfadjoint, by the above we have 
$\mathrm{F}^{k-1}=0$. Now $\mathrm{F}=0$ follows from the inductive 
assumption.
\end{proof}

\begin{proposition}\label{prop93}
Let $\mathrm{F}$ be a selfadjoint endofunctor on 
$A\text{-}\mathrm{mod}$ such that 
\begin{equation}\label{eq091}
\mathrm{F}^k\cong \mathrm{F}^m
\end{equation}
for some $k,m\in\mathbb{N}$, $k>m\geq 1$. 
\begin{enumerate}[(i)]
\item\label{prop93.1} If $k-m$ is odd, then $\mathrm{F}^2\cong 
\mathrm{F}$ (and, conversely, any $\mathrm{F}$ satisfying
$\mathrm{F}^2\cong \mathrm{F}$ obviously satisfies 
$\mathrm{F}^k\cong \mathrm{F}^m$).  
\item\label{prop93.2} If $k-m$ is even, then there is 
a decomposition $A\cong B\oplus C$ into a direct sum of algebras 
(unital or zero) and an algebra automorphism $\varphi:B\to B$
such that $\varphi^2$ is inner and 
$\mathrm{F}\cong \mathrm{F}_{\varphi}\boxplus \mathbf{0}$.
\end{enumerate}
\end{proposition}

\begin{proof}
From \eqref{eq091} it follows that 
$\mathrm{F}^{m+i(k-m)}\cong \mathrm{F}^m$ for all $i\in\mathbb{N}$.
We can choose $i$ such that $s:=i(k-m)>m$. Applying
$\mathrm{F}^{s-m}$ to $\mathrm{F}^{m+s}\cong \mathrm{F}^m$ we get
$\mathrm{F}^{2s}\cong \mathrm{F}^s$. As $\mathrm{F}^s$ is selfadjoint,
from Theorem~\ref{thm51} we obtain a decomposition $A\cong B\oplus C$ 
into a direct sum of algebras (unital or zero) such that
$\mathrm{F}^s\cong \mathfrak{p}_B$. We have
\begin{displaymath}
A\text{-}\mathrm{mod}\cong  
B\text{-}\mathrm{mod}\oplus C\text{-}\mathrm{mod}.
\end{displaymath}

\begin{lemma}\label{lem94}
We have $\mathrm{F}N=0$ for any $N\in C\text{-}\mathrm{mod}$.
\end{lemma}

\begin{proof}
We prove that $\mathrm{F}^iN=0$ by decreasing induction on $i$.
As $\mathrm{F}^s\cong \mathfrak{p}_B$, we have
$\mathrm{F}^sN=0$, which is the basis of our induction. 
For $i\in\{1,2,\dots,s-1\}$ we have, by adjunction,
\begin{displaymath}
\mathrm{Hom}_A(\mathrm{F}^iN,\mathrm{F}^iN)\cong 
\mathrm{Hom}_A(N,\mathrm{F}^{2i}N).
\end{displaymath}
From the inductive assumption we have $\mathrm{F}^{2i}N=0$ which implies
$\mathrm{Hom}_A(N,\mathrm{F}^{2i}N)=0$ and hence $\mathrm{F}^iN=0$.
\end{proof}

\begin{lemma}\label{lem95}
The functor $\mathrm{F}$ preserves $B\text{-}\mathrm{mod}$.
\end{lemma}

\begin{proof}
For $N\in B\text{-}\mathrm{mod}$ we have, by adjunction,
\begin{displaymath}
\mathrm{Hom}_A({}_CC,\mathrm{F}N)\cong 
\mathrm{Hom}_A(\mathrm{F}{}_CC,N)
\overset{\text{Lemma~\ref{lem94}}}{\cong} 0.
\end{displaymath}
The claim follows.
\end{proof}

From Lemmata~\ref{lem94} and \ref{lem95} we may write
$\mathrm{F}=\mathrm{G}_B\boxplus \mathbf{0}$, where 
$\mathrm{G}_B$ is a selfadjoint endofunctor on $B\text{-}\mathrm{mod}$. 
From $\mathrm{F}^s\cong \mathfrak{p}_B$ we obtain 
$\mathrm{G}_B^s\cong \mathrm{ID}$. By Proposition~\ref{prop2}, the
latter yields $\mathrm{G}_B\cong \mathrm{F}_{\varphi}$ for some
algebra automorphism $\varphi:B\to B$ such that $\varphi^2$ is inner.
 
Note that $\mathrm{F}_{\varphi}^2\cong \mathrm{ID}$. Therefore in the 
case when $k-m$ is odd, we must have that already 
$F_{\varphi}\cong \mathrm{ID}$, which implies that 
$\mathrm{F}\cong \mathfrak{p}_B$. It is easy
to see that $\mathfrak{p}_B$ satisfies \eqref{eq091}.

In the case when $k-m$ is even, it is easy to check that every
$\mathrm{F}_{\varphi}\boxplus \mathbf{0}$, for $\varphi$ as above, 
satisfies \eqref{eq091}. The claim follows.
\end{proof}

\section{Semisimple algebras}\label{s5}

For a semisimple algebra $A\cong\oplus_{i=1}^n\Bbbk$ there is a 
natural bijection between isomorphism classes of endofunctors on 
$A\text{-}\mathrm{mod}$ and matrices in 
$\mathrm{Mat}_{n\times n}(\mathbb{Z}_+)$. The correspondence is 
given as follows: The endofunctor $\mathrm{F}$ on
$A\text{-}\mathrm{mod}$ is sent to the matrix $M_{\mathrm{F}}$.
The inverse of this map is defined as follows:
Denote by $\Bbbk_{(i)}$, $i=1,\dots,n$, the
$i$-th simple component of the algebra $A$ (i.e. 
$A=\Bbbk_{(1)}\oplus\dots\oplus\Bbbk_{(n)}$). 
Then the matrix $X=(\mathtt{x}_{i,j})_{i,j=1,\dots,n}$
is sent to the direct sum (over all $i$ and $j$) of the
functors $\mathtt{x}_{i,j}\mathrm{ID}:\Bbbk_{(j)}\text{-}\mathrm{mod}
\to\Bbbk_{(i)}\text{-}\mathrm{mod}$. We have

\begin{proposition}\label{prop61}
Let $A\cong\oplus_{i=1}^n\Bbbk$.
\begin{enumerate}[(i)]
\item\label{prop61.1} An endofunctor  $\mathrm{F}$ on
$A\text{-}\mathrm{mod}$ is selfadjoint if and only if 
$M_{\mathrm{F}}$ is symmetric.
\item\label{prop61.2} Let $g(x),h(x)\in\mathbb{Z}_+(x)$. Then there is
a one-to one correspondence between the isomorphism classes of 
(selfadjoint) endofunctors $\mathrm{F}$ on $A\text{-}\mathrm{mod}$ 
satisfying $g(\mathrm{F})\cong h(\mathrm{F})$ and (symmetric) solutions
(in $\mathrm{Mat}_{n\times n}(\mathbb{Z}_+)$) of the matrix equation
$g(x)=h(x)$.
\end{enumerate} 
\end{proposition}

\begin{proof}
Since over $A$ simple modules are projective, 
claim \eqref{prop61.1} follows from Lemma~\ref{lem9}.
Claim \eqref{prop61.2} follows from \eqref{prop61.1}, the
complete reducibility of functors on semisimple algebras and the
previous paragraph.
\end{proof}

In light of Proposition~\ref{prop61} the problem we consider in this
paper may be viewed as a kind of a categorical generalization of the
problem of solving matrix equations. From Proposition~\ref{prop61} 
we have the following general criterion for solubility of 
functorial equations:

\begin{corollary}\label{cor62}
Let $g(x),h(x)\in\mathbb{Z}_+(x)$. Then the following conditions
are equivalent:
\begin{enumerate}[(a)]
\item\label{cor62.1} There is a finite dimensional basic 
$\Bbbk$-algebra $A$ with $n$ isomorphism classes of simple modules and 
an exact endofunctor $\mathrm{F}$ of $A\text{-}\mathrm{mod}$
such that $g(\mathrm{F})\cong h(\mathrm{F})$. 
\item\label{cor62.2} There is a matrix
$X\in\mathrm{Mat}_{n\times n}(\mathbb{Z}_+)$ such that
$g(X)=h(X)$. 
\end{enumerate}
\end{corollary}

\begin{proof}
If $A$ and $\mathrm{F}$ are as in  \eqref{cor62.1}, then 
$M_{\mathrm{F}}$ is a solution to the matrix equation 
$g(x)=h(x)$. Hence \eqref{cor62.1} implies \eqref{cor62.2}.

On the other hand, that \eqref{cor62.2} implies \eqref{cor62.1}
in the case of a semisimple algebra $A$ follows from
Proposition~\ref{prop61}. This completes the proof.
\end{proof}

Note that a tensor product of a semisimple algebra and a local 
algebra is a direct sum of local algebras. Therefore
we would like to finish this section with the following observation,
which might be used for reduction of certain classification problems to
corresponding problems over semisimple algebras.

\begin{proposition}\label{prop91}
Let $A$ be a finite-dimensional algebra and $\mathrm{F}_1,\dots,
\mathrm{F}_k$ be a collection of selfadjoint endofunctors on
$A\text{-}\mathrm{mod}$ such that the following conditions
are satisfied:
\begin{enumerate}[(a)]
\item \label{prop91.1} For every $i=1,\dots,k$ we have 
$M_{\mathrm{F}_i}=M_{\mathrm{F}_i}^t$.
\item \label{prop91.2} For some field $\mathbb{K}$ of characteristic
zero the space $\mathbb{K}\otimes_{\mathbb{Z}}
[A]$ does not contain any proper subspace 
invariant under all $[\mathrm{F}_i]$.
\end{enumerate}
Then $A$ is a direct sum of local algebras of the same dimension.
\end{proposition}

\begin{proof}
Let $C$ denote the Cartan matrix of $A$ (i.e. the matrix of multiplicities
of simple modules in projective modules). Then 
$[\mathrm{F}_i]C=C[\mathrm{F}_i]$ for all  $i=1,\dots,k$ by
Lemma~\ref{lem8} and condition \eqref{prop91.1}. Since the 
representation $\mathbb{K}\otimes_{\mathbb{Z}}[A]$ of the associative
algebra, generated by the $[\mathrm{F}_i]$, $i=1,\dots,k$, is irreducible
by \eqref{prop91.2}, from the Schur Lemma it follows that $C$ is 
a multiple of the identity matrix. The claim follows.
\end{proof}

\section{Restriction to centralizer subalgebras}\label{s7}

Let $X$ be a projective $A$-module and $B=\mathrm{End}_A(X)^{\mathrm{op}}$
(the corresponding centralizer subalgebra of $A$).
Then $X$ has the natural structure of an $A\text{-}B$-bimodule.
Denote by $\mathrm{add}(X)$ the additive closure of $X$, that is the
full subcategory of $A\text{-}\mathrm{mod}$, which consists of all modules 
$Y$, isomorphic to direct sums of (some) direct summands of $X$.
Consider the full subcategory $\mathcal{X}=\mathcal{X}_X$ of 
$A\text{-}\mathrm{mod}$, which consists of all modules $Y$ 
admitting a two step resolution
\begin{equation}\label{eq071}
X_1\to X_0 \to Y\to 0,\quad\quad
X_0,X_1\in\mathrm{add}(X).
\end{equation}
The functor $\Phi:=\mathrm{Hom}_{A}(X,{}_-):\mathcal{X}\to 
B\text{-}\mathrm{mod}$ is an equivalence, see \cite[\S~5]{Au}.

\begin{proposition}\label{prop71}
Assume that $\mathrm{F}$ is a selfadjoint endofunctor on 
$A\text{-}\mathrm{mod}$ such that $\mathrm{F}X\in \mathrm{add}(X)$.
Then the following holds:
\begin{enumerate}[(i)]
\item\label{prop71.1} The functor $\mathrm{F}$ preserves $\mathcal{X}$
and induces (via $\Phi$) a selfadjoint 
en\-do\-func\-tor $\overline{\mathrm{F}}$ 
on  $B\text{-}\mathrm{mod}$.
\item\label{prop71.2} If $g(x),h(x)\in\mathbb{Z}_+[x]$ and
$g(\mathrm{F})\cong h(\mathrm{F})$, then 
$g(\overline{\mathrm{F}})\cong h(\overline{\mathrm{F}})$.
\end{enumerate}
\end{proposition}

\begin{proof}
Applying $\mathrm{F}$ to the exact sequence \eqref{eq071} we obtain
an exact sequence
\begin{displaymath}
\mathrm{F} X_1\to \mathrm{F}X_0 \to \mathrm{F}Y\to 0.
\end{displaymath}
Here both $\mathrm{F} X_0$ and $\mathrm{F} X_1$ are in $\mathrm{add}(X)$
by assumption and hence $\mathrm{F}Y\in \mathcal{X}$. Therefore 
$\mathrm{F}$ preserves $\mathcal{X}$ and hence 
$\overline{\mathrm{F}}:=\Phi\mathrm{F}\Phi^{-1}$ is a selfadjoint 
endofunctor on  $B\text{-}\mathrm{mod}$. This proves claim \eqref{prop71.1}.
Claim \eqref{prop71.2} follows from the definition of 
$\overline{\mathrm{F}}$ by restricting any isomorphism 
$g(\mathrm{F})\cong h(\mathrm{F})$ to the subcategory $\mathcal{X}$,
which is preserved by both $g(\mathrm{F})$ and $h(\mathrm{F})$
by claim \eqref{prop71.1}. This completes the proof.
\end{proof}

\begin{corollary}\label{cor72}
Assume that $X$ is a multiplicity free direct sum of all indecomposable
projective-injective $A$-modules and $X\neq 0$. Then we have the following: 
\begin{enumerate}[(i)]
\item\label{prop72.1} Any selfadjoint endofunctor $\mathrm{F}$ 
on $A\text{-}\mathrm{mod}$ induces a selfadjoint endofunctor $\overline{\mathrm{F}}$ on  $B\text{-}\mathrm{mod}$.
\item\label{prop72.2} The map $\mathrm{F}\mapsto \overline{\mathrm{F}}$
is functorial in $\mathrm{F}$.
\end{enumerate}
\end{corollary}

\begin{proof}
From the definition of $X$ we have that the category $\mathrm{add}(X)$
is just the full subcategory of $A\text{-}\mathrm{mod}$ consistsing of
all projective-injective modules. If $\mathrm{F}$ is a selfadjoint 
endofunctor on $A\text{-}\mathrm{mod}$, then $\mathrm{F}$ preserves 
both projective and injective modules and hence preserves $\mathrm{add}(X)$.
Therefore claim \eqref{prop72.1} follows from 
Proposition~\ref{prop71}\eqref{prop71.1}. Up to conjugation with 
the equivalence $\Phi$, the map $\mathrm{F}\mapsto \overline{\mathrm{F}}$
is just the restriction map to an invariant subcategory, which is
functorial.
\end{proof}

Until the end of this section we assume that $X$ is projective-injective.
Recall (see \cite{Ta,KSX,MS4}) that $A$ is said to have the {\em double
centralizer property} for $X$ provided that there is an exact sequence
\begin{equation}\label{eq072}
{}_AA\hookrightarrow X_0\overset{\alpha}\to 
X_1,\quad\quad X_0,X_1\in\mathrm{add}(X).
\end{equation}
The name comes from the observation, see \cite{Ta}, that in this case
the actions of $A$ and $B$ on $X$ are exactly the centralizers of each
other. Examples of such situations include blocks of various generalizations
of the BGG category $\mathcal{O}$, see \cite{MS4} for details. 
The following result can be seen as a generalization of 
\cite[Theorem~1.8]{St2}, where a similar result was obtained for
projective functors on the category $\mathcal{O}$ (and its
parabolic version).

\begin{theorem}\label{prop73}
Assume that $X$ is projective-injective and that $A$ has the double
centralizer property for $X$. Then the functor 
$\mathrm{F}\mapsto \overline{\mathrm{F}}$ from 
Corollary~\ref{cor72} is full and faithful.
\end{theorem}

\begin{proof}[Proof of faithfulness.]
Let $\mathrm{F}$ and $\mathrm{G}$ be two selfadjoint endofunctors on
$A\text{-}\mathrm{mod}$ and $\xi:\mathrm{F}\to\mathrm{G}$ be a
natural transformation. Assume that 
$\overline{\xi}:\overline{\mathrm{F}}\to\overline{\mathrm{G}}$ is zero.
Since both $\mathrm{F}$ and $\mathrm{G}$ are exact, from \eqref{eq072}
we have the following commutative diagram with exact rows:
\begin{displaymath}
\xymatrix{ 
0\ar[rr] && \mathrm{F}{}_AA \ar@{^{(}->}[rr]\ar[d]^{\xi_{{}_AA}}
&&\mathrm{F}X_0\ar[rr]\ar[d]^{\xi_{X_0}}&&
\mathrm{F}X_1\ar[d]^{\xi_{X_1}}\\
0\ar[rr] && \mathrm{G}{}_AA \ar@{^{(}->}[rr]
&&\mathrm{G}X_0\ar[rr]&&\mathrm{G}X_1
}
\end{displaymath}
By assumption, $\overline{\xi}$ is zero, which means that 
both $\xi_{X_0}$ and $\xi_{X_1}$ are zero. Therefore 
$\xi_{{}_AA}$ is zero as well.

Now for any $M\in A\text{-}\mathrm{mod}$ consider the first two steps
of the projective resolution of $M$:
\begin{equation}\label{eq073}
P_1\to P_0\to M\to 0.
\end{equation}
Since both $\mathrm{F}$ and $\mathrm{G}$ are exact, from \eqref{eq073}
we have the following commutative diagram with exact rows:
\begin{displaymath}
\xymatrix{ 
\mathrm{F}P_1 \ar[rr]\ar[d]^{\xi_{P_1}}
&&\mathrm{F}P_0\ar[rr]\ar[d]^{\xi_{P_0}}&&
\mathrm{F}M\ar[d]^{\xi_{M}}\ar[rr]&&0\\
\mathrm{G}P_1 \ar[rr]&&\mathrm{G}P_0\ar[rr]&&
\mathrm{G}M\ar[rr]&&0
}
\end{displaymath}
As $\xi_{{}_AA}$ is zero by the previous paragraph
and $P_0,P_1\in\mathrm{add}({}_AA)$, we have that 
both $\xi_{P_0}$ and $\xi_{P_1}$ are zero. Therefore 
$\xi_{M}$ is zero as well. This shows that 
the natural transformation $\xi$ is zero, which establishes 
faithfullness of the functor 
$\mathrm{F}\mapsto \overline{\mathrm{F}}$.
\end{proof}

\begin{proof}[Proof of fullness.]
Let $\mathrm{F}$ and $\mathrm{G}$ be two selfadjoint endofunctors on
$A\text{-}\mathrm{mod}$ and $\xi:\overline{\mathrm{F}}\to
\overline{\mathrm{G}}$ be a natural transformation. Then we have the
following commutative diagram:
\begin{displaymath}
\xymatrix{ 
\overline{\mathrm{F}}\Phi X_0\ar[rr]^{\overline{\mathrm{F}}\Phi\alpha} 
\ar[d]_{\xi_{\Phi X_0}}
&&\overline{\mathrm{F}}\Phi X_1\ar[d]^{\xi_{\Phi X_1}}\\
\overline{\mathrm{G}}\Phi X_0\ar[rr]^{\overline{\mathrm{G}}\Phi\alpha}
&&\overline{\mathrm{G}}\Phi X_1
}
\end{displaymath}
Applying $\Phi^{-1}$ we obtain the following diagram, the solid
part of which commutes:
\begin{equation}\label{eq076}
\xymatrix{ 
\mathrm{F}{}_AA\ar@{^{(}->}[rr]\ar@{.>}[d]^{\eta}
&&\mathrm{F} X_0\ar[rr]^{\mathrm{F}\alpha} 
\ar[d]_{\Phi^{-1}\xi_{\Phi X_0}}
&&\mathrm{F} X_1\ar[d]^{\Phi^{-1}\xi_{\Phi X_1}}\\
\mathrm{G}{}_AA\ar@{^{(}->}[rr]
&&\mathrm{G} X_0\ar[rr]^{\mathrm{G}\alpha}
&&\mathrm{G} X_1
}
\end{equation}
Because of the commutativity of the solid part, the diagram extends 
uniquely to a commutative diagram by the dotted arrow $\eta$.
We claim that $\eta$ is, in fact,  a bimodule homomorphism. Indeed, 
any homomorphism $f:{}_AA\to {}_AA$ can be extended, by the injectivity
of $X$, to a commutative diagram as follows:
\begin{equation}\label{eq077}
\xymatrix{ 
{}_AA\ar@{^{(}->}[rr]\ar[d]^{f}
&&X_0\ar[rr]^{\alpha} 
\ar[d]_{f_0}
&&X_1\ar[d]^{f_1}\\
{}_AA\ar@{^{(}->}[rr]
&&X_0\ar[rr]^{\alpha}
&&X_1
}
\end{equation}
Consider the following diagram:
\begin{equation}\label{eq078}
\xymatrix@!=0.6pc{ 
&&\mathrm{Ker}(\mathrm{F}\alpha)\ar@{^{(}->}[rrrr]
\ar@{.>}[dd]|->>>>>>>>>{\eta}
\ar@{-->}[lld]_{\mathrm{F}f}
&&&&\mathrm{F} X_0\ar[rrrr]^{\mathrm{F}\alpha} 
\ar[dd]|->>>{\Phi^{-1}\xi_{\Phi X_0}}
\ar@{-->}[lld]_{\mathrm{F}f_0}
&&&&\mathrm{F} X_1\ar[dd]|->>>>>>>{\Phi^{-1}\xi_{\Phi X_1}}
\ar@{-->}[lld]_{\mathrm{F}f_1}\\ 
\mathrm{Ker}(\mathrm{F}\alpha)\ar@{^{(}->}[rrrr]\ar@{.>}[dd]^{\eta}
&&&&\mathrm{F} X_0\ar[rrrr]^>>>>>>>{\mathrm{F}\alpha} 
\ar[dd]|->>>{\Phi^{-1}\xi_{\Phi X_0}}
&&&&\mathrm{F} X_1\ar[dd]|->>>{\Phi^{-1}\xi_{\Phi X_1}}\\ 
&&\mathrm{Ker}(\mathrm{G}\alpha)\ar@{^{(}->}[rrrr]
\ar@{-->}[lld]_{\mathrm{G}f}
&&&&\mathrm{G} X_0\ar[rrrr]^>>>>>{\mathrm{G}\alpha}
\ar@{-->}[lld]^{\mathrm{G}f_0}
&&&&\mathrm{G} X_1\ar@{-->}[lld]^{\mathrm{G}f_1}
\\
\mathrm{Ker}(\mathrm{G}\alpha)\ar@{^{(}->}[rrrr]
&&&&\mathrm{G} X_0\ar[rrrr]_{\mathrm{G}\alpha}
&&&&\mathrm{G} X_1
}
\end{equation}
The upper face of the diagram \eqref{eq078} commutes as it coincides
with the image of the commutative diagram \eqref{eq077} under $\mathrm{F}$.
Similarly, the lower face of the diagram \eqref{eq078} commutes as it 
coincides with the image of the commutative diagram \eqref{eq077} 
under $\mathrm{G}$. The front and the back faces coincide with
\eqref{eq076} and hence commute. The right and the middle square sections
commute as $\xi$ is a natural transformation. This implies that 
the whole diagram commutes, showing that $\eta$ is indeed a bimodule
map from  $\mathrm{F}$ to $\mathrm{G}$.

This means that $\eta$ defines a 
natural transformation from $\mathrm{F}$ to $\mathrm{G}$.
By construction, we have $\xi=\overline{\eta}$, which proves that 
the functor $\mathrm{F}\mapsto \overline{\mathrm{F}}$ is full.
\end{proof}

Unfortunately, the functor $\mathrm{F}\mapsto \overline{\mathrm{F}}$ 
from Corollary~\ref{cor72} is not dense (in particular, not an 
equivalence between the monoidal categories of selfadjoint 
endofunctors on $A\text{-}\mathrm{mod}$ and $B\text{-}\mathrm{mod}$)
in the general case. Let $\mathrm{G}$ be a selfadjoint 
endofunctors on $B\text{-}\mathrm{mod}$ and assume that 
$\mathrm{G}=\overline{\mathrm{F}}$ for some selfadjoint 
endofunctors on $A\text{-}\mathrm{mod}$. Then from
\eqref{eq072} we have 
\begin{equation}\label{eq0705}
\mathrm{F}{}_AA=\mathrm{Ker}(\Phi^{-1}\mathrm{G}\Phi\alpha)
\end{equation}
(as a bimodule, with the induced action on morphisms), 
which uniquely defines the functor $\mathrm{F}$ (see
\cite[Chapter~II]{Ba}). However, here is an example of $A$, $X$
and $\mathrm{G}$ for which the bimodule 
$\mathrm{Ker}(\Phi^{-1}\mathrm{G}\Phi\alpha)$ defines only a
right exact (and hence not selfadjoint) functor:

\begin{example}\label{ex79}
{\rm 
Let $A$ be the algebra of the following quiver with relations:
\begin{displaymath}
\xymatrix{ 
1\ar@(ur,dr)[]^{x} && 2\ar@/^/[r]^{a} & 3\ar@/^/[l]^{b}
}\quad\quad\quad ab=x^2=0
\end{displaymath}
The indecomposable projective $A$-modules look as follows:
\begin{displaymath}
P_1:\quad
\xymatrix{
1\ar[d]^x\\1
}\quad\quad\quad\quad
P_2:\quad
\xymatrix{
2\ar[d]^a\\3\ar[d]^b\\2
}\quad\quad\quad\quad
P_3:\quad
\xymatrix{
3\ar[d]^b\\2
}
\end{displaymath}
The modules $P_1$ and $P_2$ are injective, so we take $X=P_1\oplus P_2$
and have that $B$ is isomorphic to the algebra of the following quiver 
with relations:
\begin{equation}\label{eq0701}
\xymatrix{ 
1\ar@(ur,dr)[]^{x} && 2\ar@(ur,dr)[]^{y}
}\quad\quad\quad x^2=y^2=0
\end{equation}
(here $y=ba$). The double centralizer property is guaranteed by the
fact that the first two steps of the injective coresolution of $P_3$
are as follows:
\begin{displaymath}
0\to P_3\to P_2\overset{\beta}{\to} P_2,
\end{displaymath}
where $\beta$ is the right multiplication with the element $ba$.
Let $\varphi:B\to B$ be the involutive automorphism of $B$ given by
the automorphism of the quiver \eqref{eq0701} swapping the vertices.
Then $\mathrm{G}:={}_{\varphi}B\otimes_{B}{}_-$ is a selfadjoint 
autoequivalence of $B\text{-}\mathrm{mod}$ (see Proposition~\ref{prop2}). 
Assume that $\mathrm{F}$ is a right exact endofunctor on 
$A\text{-}\mathrm{mod}$ given by \eqref{eq0705}. Then the restriction of 
$\mathrm{F}$ to $\mathrm{add}(X)$ is isomorphic to $\mathrm{G}$, which 
implies $\mathrm{F}P_1\cong P_2$. For $i=1,2,3$ we denote by $L_i$ the 
simple head of $P_i$. Applying $\mathrm{F}$ to the exact
sequence $P_1\to P_1\tto L_1$, we get the exact sequence
$P_2\to P_2\tto \mathrm{F}L_1$, which implies that the module
$\mathrm{F}L_1$ is isomorphic to the following module:
\begin{displaymath}
N:\quad \xymatrix{
2\ar[d]^a\\3
}
\end{displaymath}
Now, applying $\mathrm{F}$ to the short exact
sequence $L_1\hookrightarrow P_1\tto L_1$ we obtain the sequence
\begin{displaymath}
N\to P_2\tto N, 
\end{displaymath}
which is not exact. This means that $\mathrm{F}$ is not exact and thus
cannot be selfadjoint.
}
\end{example}

It would be interesting to know when the functor 
$\mathrm{F}\mapsto \overline{\mathrm{F}}$ 
from Corollary~\ref{cor72} is dense.

\section{Invariant Serre subcategories and quotients}\label{s8}

For $S\subset \{1,2,\dots,n\}$ set $S'=\{1,2,\dots,n\}\setminus S$
and let $\mathcal{N}_S$
denote  the full subcategory of $A\text{-}\mathrm{mod}$, which 
consists of all modules $N$ for which $[N:L_i]\neq 0$ implies
$i\in S$. Then $\mathcal{N}_S$ is a Serre subcategory of 
$A\text{-}\mathrm{mod}$ and, moreover, any Serre subcategory of 
$A\text{-}\mathrm{mod}$ equals $\mathcal{N}_S$ for some 
$S$ as above. Both $\mathcal{N}_S$ and the quotient $\mathcal{Q}_S:=
A\text{-}\mathrm{mod}/\mathcal{N}_S$ are abelian categories.
Recall (see e.g. \cite[Chapter~III]{Ga} or \cite[Chapter~15]{Fa}) 
that the quotient $\mathcal{Q}_S$ has the same objects as 
$A\text{-}\mathrm{mod}$ and for objects $M,N$ we have
\begin{displaymath}
\mathrm{Hom}_{\mathcal{Q}_S}(M,N)=\lim_{\longrightarrow}
\mathrm{Hom}_A(M',N/N'),
\end{displaymath}
where $M'\subset M$ and $N'\subset N$ are such that 
$M/M',N'\in \mathcal{N}_S$. As we are working with finite dimensional
modules, the space  $\mathrm{Hom}_{\mathcal{Q}_S}(M,N)$ can be 
alternatively described as follows: For a module $M$ let $M^-$ denote 
the smallest submodule of
$M$ such that $M/M^-\in \mathcal{N}_S$ and $M^+$ denote the largest 
submodule of $M$ such that $M^+\in \mathcal{N}_S$. Then we have
\begin{displaymath}
\mathrm{Hom}_{\mathcal{Q}_S}(M,N)=
\mathrm{Hom}_A((M^-+M^+)/M^+,(N^-+N^+)/N^+).
\end{displaymath}

For $S\subset \{1,2,\dots,n\}$ define $P_S:=\oplus_{i\in S}P_i$
and $B_S:=\mathrm{End}_A(P_S)^{\mathrm{op}}$. If $S$ is nonempty,
let $I_S$ denote the trace of the module $P_{S'}$ in ${}_AA$. 
Then $I_S$ is obviously an 
ideal in $A$, so we can define the quotient algebra $D_S:=A_S/I_S$.

\begin{proposition}\label{prop81}
For any $N\in \mathcal{N}_S$ we have $I_S N=0$, so such $N$ becomes
a $D_S$-module. This defines an equivalence $\mathcal{N}_S\cong 
D_S\text{-}\mathrm{mod}$.
\end{proposition}

\begin{proof}
The quotient map $A\tto D_S$ defines a full and faithful embedding of
$D_S\text{-}\mathrm{mod}$ into $A\text{-}\mathrm{mod}$ and the image
of this embedding consists exactly of $N\in A\text{-}\mathrm{mod}$ 
such that $I_SN=0$. 

If $N\in \mathcal{N}_S$, then $\mathrm{Hom}_A(P_{S'},N)=0$ by the
definition of $\mathcal{N}_S$, which implies $I_S N=0$. Conversely,
if $N\in A\text{-}\mathrm{mod}$ is such that $N\neq 0$, $I_S N=0$, then 
$\mathrm{Hom}_A(P_{S'},N)=0$ and hence every composition subquotient
of $N$ is isomorphic to some $L_i$, $i\in S$. This means that
$\mathcal{N}_S$ coincides with the image of $D_S\text{-}\mathrm{mod}$ 
in $A\text{-}\mathrm{mod}$ and the claim follows.
\end{proof}

\begin{proposition}\label{prop82}
Let $S\subsetneq \{1,2,\dots,n\}$. Then we have equivalences
$\mathcal{Q}_S\cong \mathcal{X}_{P_{S'}}\cong B_{S'}\text{-}\mathrm{mod}$.
\end{proposition}

\begin{proof}
That $\mathcal{X}_{P_{S'}}$ is equivalent to $B_{S'}\text{-}\mathrm{mod}$
follows from \cite[Chapter~II]{Ba}
(see also \cite[\S~5]{Au}). Let us show that the embedding of
$\mathcal{X}_{P_{S'}}$ to $A\text{-}\mathrm{mod}$ induces an equivalence
$\mathcal{Q}_S\cong \mathcal{X}_{P_{S'}}$ via the canonical quotient map 
$A\text{-}\mathrm{mod}\tto \mathcal{Q}_S$. Let 
$\Psi:\mathcal{X}_{P_{S'}}\hookrightarrow A\text{-}\mathrm{mod}
\tto \mathcal{Q}_S$ denote the corresponding  functor.
 
If $M\in \mathcal{X}_{P_{S'}}$ and $M'\subset M$, then $M/M'$ is a 
quotient of some module from $\mathrm{add}(P_{S'})$. Hence
$M/M'\not \in \mathcal{N}_S$ unless $M/M'=0$. 

\begin{lemma}\label{lem84}
For $M\in \mathcal{X}_{P_{S'}}$ we have 
$\mathrm{Ext}_A^1(M,Z)=0$ for any $Z\in \mathcal{N}_S$.
\end{lemma}

\begin{proof}
Let $X_1\to X_0\tto M$ be the first two steps of the projective 
resolution of $M$, given by \eqref{eq071}. Then 
$\mathrm{Hom}_A(X_1,Z)=0$ (as the head of $X_1$ contains only
$L_j$, $j\in S'$, while all composition subquotients of $Z$ are
of the form $L_i$, $i\in S$) and the claim follows.
\end{proof}

If $M\in \mathcal{X}_{P_{S'}}$
and $N'\subset N$ is such that $N'\in \mathcal{N}_S$, then 
$\mathrm{Hom}_A(M,N')=0$ and $\mathrm{Ext}_A^1(M,N')=0$
(the latter by Lemma~\ref{lem84}). Therefore 
$\mathrm{Hom}_A(M,N)\cong \mathrm{Hom}_A(M,N/N')$.
Combining this with the paragraph before Lemma~\ref{lem84} we 
have $\mathrm{Hom}_A(M,N)=\mathrm{Hom}_A(M',N/N')$ in the
case $M,N\in \mathcal{X}_{P_{S'}}$. This yields 
\begin{displaymath}
\mathrm{Hom}_{\mathcal{Q}_S}(M,N)=
\mathrm{Hom}_A(M,N)\quad \text{ for all }\quad 
M,N\in \mathcal{X}_{P_{S'}}.
\end{displaymath}
This means that the functor $\Psi$ is full and faithfull. It is
left to prove that $\Psi$ is dense.

Let $N$ be an $A$-module and $N'$ be the trace of $P_{S'}$ in $N$.
Take a projective cover $X_0\tto N'$, where $X_0\in\mathrm{add}(P_{S'})$,
let $Q$ be the kernel of this epimorphism and $Q'$ be the trace of 
$P_{S'}$ in $Q$. Define $M=X_0/Q'$ and $M'=Q/Q'\subset M$. Then
$M',N/N'\in \mathcal{N}_S$ and $M/M'\cong X_0/Q\cong N'$ by definition. 
Let $\varphi:M\to N$ be the composition of the natural maps 
$M\to N'\hookrightarrow N$. Let $\psi:N'\to M/M'$ be the inverse
of the natural isomorphism $M/M'\overset{\sim}{\rightarrow} N'$. Then both 
$\varphi\in \mathrm{Hom}_{\mathcal{Q}_S}(M,N)$ and
$\psi\in \mathrm{Hom}_{\mathcal{Q}_S}(N,M)$ and it is straightforward to 
check that $\varphi$ and $\psi$ are mutually inverse isomorphisms.
This means that $N$ is isomorphic in $\mathcal{Q}_S$ to 
$M\in \mathcal{X}_{P_{S'}}$ and hence the functor $\Psi$ is dense.
This completes the proof.
\end{proof}

Proposition~\ref{prop82} can be deduced from the results described
in \cite[Chapter~15]{Fa}. However, it is shorter to prove it in the
above form than to introduce all the notions and notation necessary 
for application of  \cite[Chapter~15]{Fa}.
The correspondence $N\mapsto M$ from the last paragraph of the 
proof of Proposition~\ref{prop82} is functorial. The module 
$M$ is called the {\em partial coapproximation} of $N$ with respect to
$\mathcal{X}_{P_{S'}}$, see \cite[2.5]{KM} for details.
From Proposition~\ref{prop82} it follows that 
$\mathcal{S}$-subcategories of the BGG category $\mathcal{O}$ associated 
with parabolic $\mathfrak{sl}_2$-induction (see \cite{FKM1,FKM2}) can 
been regarded as quotients of blocks of the usual category 
$\mathcal{O}$ modulo the corresponding parabolic subcategory (in the 
sense of \cite{RC}). In the general case $\mathcal{S}$-subcategories 
of $\mathcal{O}$ are quotient categories as well (however, modulo a
subcategory, which properly contains the corresponding parabolic 
subcategory). In fact, the latter can be deduced combining
several known results from the literature (\cite{BG}, \cite[Kapitel~6]{Ja} 
and \cite{KoM}). 

\begin{corollary}\label{cor83}
Let $\mathrm{F}$ be a selfadjoint endofunctor on
$A\text{-}\mathrm{mod}$ and $S\subsetneq \{1,2,\dots,n\}$
be such that the linear span of $[L_i]$, $i\in S$, in invariant
under $[\mathrm{F}]$.
\begin{enumerate}[(i)]
\item\label{cor83.1} The functor $\mathrm{F}$ preserves the category
$\mathcal{N}_S$ and hence induces, via restriction and the equivalence from
Proposition~\ref{prop81}, a selfadjoint en\-do\-func\-tor $\hat{\mathrm{F}}$ 
on $D_S\text{-}\mathrm{mod}$.
\item\label{cor83.2} The functor $\mathrm{F}$ preserves the category
$\mathcal{X}_{P_{S'}}$ and hence induces, via restriction and the
equivalence from Proposition~\ref{prop82}, a selfadjoint en\-do\-func\-tor 
$\overline{\mathrm{F}}$  on $B_{S'}\text{-}\mathrm{mod}$.
\item\label{cor83.3} If $g(x),h(x)\in\mathbb{Z}_+[x]$ and
$g(\mathrm{F})\cong h(\mathrm{F})$, then 
$g(\hat{\mathrm{F}})\cong h(\hat{\mathrm{F}})$ and
$g(\overline{\mathrm{F}})\cong h(\overline{\mathrm{F}})$.
\end{enumerate}
\end{corollary}

\begin{proof}
The functor  $\mathrm{F}$ preserves the category
$\mathcal{N}_S$ by our assumptions and claim \eqref{cor83.1}
follows. 

For $i\in S$ we have
\begin{displaymath}
\mathrm{Hom}_A(\mathrm{F}P_{S'},L_i)\cong 
\mathrm{Hom}_A(P_{S'},\mathrm{F}L_i)=0
\end{displaymath}
as $\mathrm{F}L_i\in \mathcal{N}_S$ by claim \eqref{cor83.1}.
This means that $\mathrm{F}P_{S'}\in\mathrm{add}(P_{S'})$ and
claim \eqref{cor83.2} follows from Proposition~\ref{prop71}.

Any isomorphism $g(\mathrm{F})\cong h(\mathrm{F})$ induces, by
restriction to $\mathcal{N}_S$ and $\mathcal{X}_{P_{S'}}$,
isomorphisms $g(\hat{\mathrm{F}})\cong h(\hat{\mathrm{F}})$ and
$g(\overline{\mathrm{F}})\cong h(\overline{\mathrm{F}})$, respectively.
This proves claim \eqref{cor83.3} and completes the proof.
\end{proof}

From the proof of Corollary~\ref{cor83} it follows that in the case when
a selfadjoint endofunctor $\mathrm{F}$ on $A\text{-}\mathrm{mod}$
preserves the category $\mathrm{add}(P_{S})$ for some nonempty
$S\subset \{1,2,\dots,n\}$, then $\mathrm{F}$ preserves the category
$\mathcal{N}_{S'}$ as well.

\vspace{0.5cm}

\noindent
T.A.: Department of Mathematics, {\AA}rhus University, DK- 8000,
{\AA}rhus C, DENMARK, e-mail: {\tt agerholm\symbol{64}imf.au.dk}
\vspace{0.2cm}

\noindent
V.M.: Department of Mathematics, Uppsala University, SE 471 06,
Uppsala, SWEDEN, e-mail: {\tt mazor\symbol{64}math.uu.se}

\end{document}